\documentclass[11pt,reqno]{article}
\usepackage{graphicx}
\usepackage{amsmath}
\usepackage{amsthm}
\usepackage{latexsym,bm}
\usepackage{amssymb}
\usepackage{indentfirst}
\newtheorem{thm}{Theorem}[section]
\newtheorem{cor}[thm]{Corollary}
\newtheorem{lem}[thm]{Lemma}

\newtheorem{rem}[thm]{Remark}
\newtheorem{exm}{Example}

\numberwithin{equation}{section}

\begin{document}

\title{Morse-Novikov cohomology of almost nonnegatively curved manifolds}

\author{Xiaoyang Chen\footnotemark[1]}
\renewcommand{\thefootnote}{\fnsymbol{footnote}}
\footnotetext[1]{School of Mathematical Sciences, Institute for Advanced Study, Tongji University, Shanghai, China. email: $xychen100@tongji.edu.cn$.}
 \date{}
\maketitle

\begin{abstract}
Let $M^n$ be a closed manifold of almost nonnegative sectional curvature and nonzero first de Rham cohomology group.
Using a topological argument, we show that the Morse-Novikov cohomology group $H^p(M^n,\theta)$ vanishes for any $p$
and $[\theta] \in H^1_{dR} (M^n), [\theta] \neq 0$.
Based on a new integral formula, we also show that a similar result holds for a closed manifold of almost nonnegative Ricci curvature
under the additional assumption that its curvature operator is uniformly bounded from below.
\end{abstract}

\section{Introduction}

Let $M^n$ be a smooth manifold and $\theta$ a real valued closed one form on $M^n$. Set $\Omega^p(M^n)$ the space of
real smooth $p$-forms and define $d_{\theta}: \Omega^p(M^n) \rightarrow \Omega^{p+1}(M^n)$ as $d_{\theta}   \alpha=d \alpha+ \theta \wedge \alpha$ for
$\alpha \in \Omega^p(M^n)$. Then we have a complex
$$\cdots \rightarrow \Omega^{p-1}(M^n)  \xrightarrow{d_{\theta}}  \Omega^{p}(M^n) \xrightarrow{d_{\theta}}  \Omega^{p+1}(M^n)  \rightarrow \cdots $$
whose cohomology $H^p (M, \theta)=H^p (\Omega^*(M^n), d_{\theta})$ is called the $p$-th Morse-Novikov cohomology group of $M^n$ with respect to $\theta$.
If $\theta_1, \theta_2 $ are two representatives in the cohomology class $[\theta]$, then $H^p ( M, \theta_1 )\simeq H^p (M, \theta_2)$.
Hence $H^p (M, \theta)$ only depends on the de Rham cohomology class of $\theta$. This cohomology shares many properties with the ordinary de Rham
cohomology. See \cite{HR, O, Paz} and section $2$ for details.

\par If $[\theta]=0$, the Novikov cohomology group $H^p(M, \theta)$ is isomorphic to the de Rham cohomology
group $H^p_{dR}(M^n)$. There are lots of work relating de Rham cohomology to curvature properties of Riemannian manifolds. See for example \cite{P}.
In particular, a celebrated theorem of Gromov says that the Betti number of a closed manifold with almost nonnegative sectional curvature
is bounded above by a constant depending only the dimension of the manifold \cite{Gr}. Here we say that a Riemannian manifold $M^n$ has
almost nonnegative sectional curvature if it admits a sequence of Riemannian metrics $g_i$ such that
$$sec (g_i) \geq -\frac{1}{i}$$
$$D(g_i)\leq 1,$$
where $sec (g_i)$ is the sectional curvature of $g_i$ and $D(g_i)$ is the diameter of $g_i$.

\par However, there are quite few work discussing the relationship between Morse-Novikov cohomology $H^p (M, \theta)$ and curvature when $[\theta] \neq 0$.
This paper is trying to make an attempt towards this direction. Our first result is the following theorem.
\begin{thm} \label{main}
Let $M^n$ be a closed Riemannian manifold of almost nonnegative sectional curvature and nonzero first de Rham
cohomology group, then the Morse-Novikov cohomology $H^p (M, \theta)=0$ for any  $p$ (including $p=0$) and any $[\theta] \in H^1_{dR}(M^n), [\theta] \neq 0$.
\end{thm}

From the work in \cite{FY92, KPT}, we know that a closed Riemannian manifold $M^n$ of almost nonnegative sectional curvature is an almost nilpotent space. Namely,
there is a finite cover of $M^n$, denoted by $\hat{M^n}$, such that $\pi_1(\hat{M^n})$ is a nilpotent group that operates nilpotently on $\pi_k(\hat{M^n})$ for every $k \geq 2$.
Recall that an action by automorphisms of a group $G$ on an abelian group $V$ is called nilpotent if $V$ admits a finite sequence of $G$-invariant subgroups
$$V=V_0 \supset V_1 \supset V_2 \supset \cdots \supset V_k =0$$
such that the induced action of $G$ on $V_j /V_{j+1}$ is trivial for any $j$. Now Theorem \ref{main} is a consequence of the following topological result.

\begin{thm} \label{nice}
Let $M^n$ be a smooth manifold with nonzero first de Rham cohomology group. If $M^n$ is an almost nilpotent space,
then the Morse-Novikov cohomology $H^p (M, \theta)=0$ for any $p$ and any $[\theta] \in H^1_{dR}(M^n), [\theta] \neq 0$.
\end{thm}

 For a smooth manifold which is $not$ an almost nilpotent space, its Morse-Novikov cohomology does  not  necessarily vanish as the following example shows.
\begin{exm} \label{exm} \cite{KPT}
Let $h: \mathbb{S}^3 \times \mathbb{S}^3 \rightarrow \mathbb{S}^3 \times \mathbb{S}^3 $ be defined by
$$ h: (x,y) \rightarrow (xy, yxy).$$
This map is a diffeomorphism with inverse given by
$$h^{-1}: (u,v) \rightarrow (u^2 v^{-1}, vu^{-1}).$$
Let $M$ be the mapping torus of $h$.  Then $M$ has the structure of a fiber bundle:
$$\mathbb{S}^3 \times \mathbb{S}^3 \rightarrow M \rightarrow \mathbb{S}^1.$$
 The induced map $h^{*,3}$ on $H^3_{dR}( \mathbb{S}^3 \times \mathbb{S}^3)$ is given by the matrix
 \begin{equation}
A_h=\left( \begin{array}{ll}{1} & {1} \\ {1} & {2}\end{array}\right)
\end{equation}
Notice that the eigenvalues of $A_h$ are different from $1$ in absolute value. Hence $M^n$ is $not$ an almost nilpotent space.
 Let $\lambda$ be a eigenvalue of $A_h$ with $\lambda=e^{-t}, t\neq 0, t\in \mathbb{R}$ and $\theta$ a generator of $H^1_{dR}(M)$.  We claim that $H^3(M, t \theta)\neq 0$.
To see this, observe that $t \theta$ defines a linear representation of the fundamental group of $M$:
$$\rho_t: \pi_1(M) \rightarrow GL(1, \mathbb{C})= \mathbb{C}^*, [\gamma] \mapsto e^{t \int_{\gamma} \theta}. $$
The representation $\rho_t$ defines a complex rank one local system $\mathbb{C}_{\rho_t}$ over $M^n$ \cite{Di}. We denote
by  $H^p(M^n, \mathbb{C}_{\rho_t})$ the $p$-th cohomology group of $M^n$ with coefficients in this local system.
By Theorem  \ref{iso} in section $2$, for any $p$, we have
$$H^p (M, t \theta)\simeq H^p (M^n, \mathbb{C}_{\rho_t}).$$
On the other hand, by Wang's exact sequence in Proposition 6.4.8 in \cite{Di} page 212, we have
$$\dim_{\mathbb{C}} H^p(M^n, \mathbb{C}_{\rho_t})=\dim_{\mathbb{C}} \ker(h^{*,p}-e^{-t}Id) + \dim_{\mathbb{C}} coker (h^{*, p-1}-e^{-t}Id),$$
where $h^{*, p}: H^p(\mathbb{S}^3 \times \mathbb{S}^3, {\mathbb{C}}) \rightarrow H^p (\mathbb{S}^3 \times \mathbb{S}^3, {\mathbb{C}})$ is
the linear map induced by $h$. As $e^{-t}$ is an eigenvalue of $h^{*,3}$,
we see that $\dim_{\mathbb{C}} \ker( h^{*,3}-e^{-t} Id)>0$ and $H^3(M, t \theta)\neq 0$.
\end{exm}

By Theorem  \ref{basic} in section $2$, we see that $\sum_{p=0}^n (-1)^p dim H^p (M^n, \theta)$ is equal to the Euler characteristic number of $M^n$. Hence we get the following
\begin{cor}
Let $M^n$ be a smooth manifold with nonzero first de Rham cohomology group. If $M^n$ is an almost nilpotent space, then its Euler characteristic number vanishes.
\end{cor}

Theorem \ref{main} fails for closed manifolds of almost nonnegative Ricci curvature. Recall that a Riemannian manifold has almost nonnegative Ricci curvature
if it admits a sequence of Riemannian metrics $g_i$ such that
$$Ric (g_i) \geq -\frac{n-1}{i}$$
$$D(g_i) \leq 1,$$
where $Ric(g_i)$ is the Ricci curvature of $g_i$ and $D(g_i)$ is the diameter of $g_i$.  Let $M^4$ be the manifold performing surgery along a meridian curve in
$T^4$, i.e, removing a tubular neighborhood of the curve and attaching a copy of $D^2 \times S^2$. In \cite{An92}, Anderson showed that
$M^4$ admits a sequence of Riemannian metrics $g_i$ such that
$$|Ric (g_i)| \leq \frac{n-1}{i}$$
$$D(g_i) \leq 1.$$
Moreover, its fundamental group is isomorphic to $\mathbb{Z}^3$ and its Euler characteristic number is nonzero.
For any $[\theta] \in H^1_{dR}(M^4), [\theta] \neq 0$, by Theorem \ref{basic} and Theorem \ref{h1} in section $2$, we get $H^p (M^4, \theta)=0$ for $p \neq 2$ and $H^2(M^4, \theta)\neq 0$.
However, the sectional curvature of $g_i$ constructed by Anderson can $not$ have a uniform lower bound. Otherwise,
there will be also an upper bound of the sectional curvature and by Theorem $1$ in \cite{Ya}, $M^4$ will fiber over $S^1$ which is impossible by the construction.
 In particular, the curvature operator of $g_i$ can $not$ have a uniform lower bound.
By the following Theorem \ref{ricci} and its Corollary \ref{and}, $M^4$ in fact can $not$ admit a sequence of Riemannian metrics $g_i$
of almost nonnegative Ricci curvature with curvature operator uniformly bounded from below.

\begin{thm} \label{ricci}
Let $M^n$ be a closed Riemannian manifold with nonzero first de Rham
cohomology group and admits a sequence of Riemannian metrics $g_i$ such that
$$Ric (g_i) \geq -\frac{n-1}{i}$$
$$D(g_i) \leq 1.$$
If the curvature operator of $g_i$ is uniformly bounded from below by $-Id$, then for any $[\theta] \in H^1_{dR}(M^n), [\theta] \neq 0$,
there exists some $t \in \mathbb{R}, t \neq 0$ such that $H^p (M, t \theta)=0$ for any $p$, where $H^p (M, t \theta)$ is the Morse-Novikov cohomology group with respect to
$t \theta$.
\end{thm}

\begin{cor} \label{and}
Let $M^n$ be a closed Riemannian manifold with nonzero first de Rham
cohomology group. If $M^n$ admits a sequence of Riemannian metrics of almost nonnegative Ricci curvature with curvature operator uniformly bounded from below,
then the Euler characteristic number of $M^n$ vanishes.
\end{cor}

It has been known that the fundamental group of a closed manifold $M^n$ of almost nonnegative Ricci curvature is almost nilpotent \cite{KW}. By Theorem \ref{h1},
$H^1 (M, \theta)=0$ for any $[\theta] \neq 0$ without any additional assumption. See \cite{KL} for related work on noncollapsed  almost Ricci flat manifolds.
\par
Finally, we point out that for a closed Riemannian manifold $M^n$ of nonnegative Ricci curvature and nonzero first de Rham
cohomology group, then the Morse-Novikov cohomology $H^p (M, \theta)=0$ for any $p$ and $[\theta] \in H^1_{dR}(M^n), [\theta] \neq 0$.
This follows from the Cheeger-Gromoll splitting theorem \cite{CG} and Theorem \ref{basic}.

\par The proof of Theorem \ref{nice} is based on Cartan-Leray spectral sequence on equivalent homology \cite{Br}.
By passing to a finite cover, we can assume that $M^n$ is a nilpotent space.
The closed one form $\theta$ on $M^n$ defines a linear representation of the fundamental group of $M^n$:
$$\rho: \pi_1(M^n) \rightarrow GL(1, \mathbb{C})= \mathbb{C}^*, [\gamma] \mapsto e^{\int_{\gamma} \theta}. $$
The representation $\rho$ defines a complex rank one local system $\mathbb{C}_{\rho}$ over $M^n$ \cite{Di}. We denote
by $H^p(M^n, \mathbb{C}_{\rho})$ the $p$-th cohomology group of $M^n$ with coefficients in this local system.
 By Theorem \ref{iso} in section $2$, for any $p$, we have
$$H^p(M^n, \theta) \simeq H^p(M^n, \mathbb{C}_{\rho}).$$
By duality, it suffices to show that $H_p(M^n, \mathbb{C}_{\rho})=0$, where $H_p(M^n, \mathbb{C}_{\rho})$
is the $p$-th homology group of $M^n$ with coefficients in this local system.
Let $\pi: \widetilde{M}^n \rightarrow M^n$ be the universal cover of $M^n$.
By the Cartan-Leray spectral sequence \cite{Br}, we have
\begin{equation} \label{sp}
E^2_{kl}=H_k(\pi_1(M^n), H_l(\widetilde{M}^n,  \mathbb{C})) \Rightarrow H_{k+l}(M^n, \mathbb{C}_{\rho}),
\end{equation}
where $H_k(\pi_1(M^n), H_l(\widetilde{M}^n, \mathbb{C}))$ is the $k$-th homology group of $\pi_1(M^n)$ with coefficients in the
$\pi_1(M^n)$-module $H_l(\widetilde{M}^n,  \mathbb{C})$.
Then we prove by induction to get the vanishing of $H_p(M^n, \mathbb{C}_{\rho})$.

\par  The proof of Theorem \ref{ricci} is based on Hodge theory of Morse-Novikov cohomology.
Let $d^*$ be the formal $L^2$ adjoint of $d$ with respect to the Riemannian metric $g_i$. We can also define an operator $d_{\theta}^*$ as
the formal $L^2$ adjoint of $d_{\theta}$ with respect to $g_i$. Further, $\Delta_{\theta}=d_{\theta} d_{\theta}^*+ d_{\theta}^*d_{\theta}$
 is the corresponding Laplacian. These operators are lower-order perturbations of the corresponding operators in the usual Hodge-de Rham theory
and therefore have much the same analytic properties. For example, the usual proof of the Hodge decomposition theorem goes through,
and one obtains an orthogonal decomposition
$$\Omega^p(M^n)=\mathcal{H}^p(M^n) \oplus d_{\theta} (\Omega^{p-1}(M^n))\oplus d_{\theta}^* (\Omega^{p+1}(M^n)),$$
where $\mathcal{H}^p(M^n)$ is the space of $\Delta_{\theta}$ harmonic forms, which is isomorphic to $H^p(M^n, \theta)$.
\par By Hodge theory, for each $i$ we can choose a harmonic form $\theta_i$ in the cohomology class $[\theta]$.
Let $V(g_i)$ be the volume of $(M^n, g_i)$, $dV_i$  the volume form of $g_i$ and $X_i$ the dual vector field of $\theta_i$
defined by $g_i(X_i, Y)=\theta (Y).$  Set $t_i=(\frac{V(g_i)}{\int_{M^n} |X_i|^2 dV_i})^{1/2}>0$. Choose a $\Delta_{t_i \theta_i}$
harmonic form $\alpha_i$ in $H^p(M^n, t_i\theta_i)$.
The idea is to show that $\alpha_i \equiv 0$ for sufficiently large $i$, which relies on the following crucial integral inequality proved in Corollary \ref{key2}.
\begin{equation} \label{point}
\int_{M^n}  t_i^2 |X_i|^2 |\alpha_i|^2 dV_i \leq C_n \int_{M^n}  (t_i |\nabla X_i|+ t_i|div (X_i)|)|\alpha_i|^2 dV_i
\end{equation}
for some constant $C_n$ depending only on $n$.
\par  As $Ric(g_i) \geq -\frac{n-1}{i}$,
 applying Bochner formula to $X_i$, we get
\begin{equation} \label{001}
\int_{M^n} |\nabla X_i|^2 dV_i \leq \frac{n-1}{i} \int_{M^n} |X_i|^2 dV_i.
\end{equation}
 Combing \ref{point} and \ref{001}, for sufficiently large $i$ we will show
 $$ \int_{M^n} |\alpha_i|^2 dV_i \leq  \frac{1}{2}\int_{M^n} |\alpha_i|^2 dV_i.$$
Hence $\alpha_i \equiv 0$.  See section $5$ for details.

\section*{Acknowledgements} The author is partially supported by National Natural Science Foundation of China No.11701427
and Institute for Advanced Study, Tongji University no. 8100141347. He thanks
Professor Binglong Chen, John Lott and Andrei Pajitnov for helpful discussions.

\section{Basic properties of Morse-Novikov cohomology}
In this section we collect some basic properties of Morse-Novikov cohomology.
\begin{thm} \label{basic}
Let $M^n$ be a compact $n$-dimensional manifold and $\theta$ a closed one form on $M^n$. Then:
\\
\\(1) If $\theta' =\theta + df, f \in C^{\infty}(M^n, \mathbb{R})$, then for any $p$, we have $H^p(M^n, \theta') \simeq H^p(M^n, \theta)$
and the isomorphism is given by the map $[\alpha] \mapsto [e^f \alpha]$;
\\
\\(2) If $[\theta] \neq 0$ and $M^n$ is connected and orientable, then $H^0(M^n, \theta)$ and  $H^n(M^n, \theta)$ vanish.
 Moreover, the integration $\int: H^p (M^n, \theta)\times H^{n-p}(M^n, -\theta), (\alpha, \beta) \mapsto \int_{M^n} \alpha \wedge \beta$
 induces an isomorphism $H^p (M^n, \theta)\simeq (H^{n-p}(M^n, -\theta))^*$.
\\
\\ (3) $\sum_{p=0}^n (-1)^p dim H^p (M^n, \theta)$ is equal to the Euler characteristic number of $M^n$;
\\
\\ (4) If $N^d$ be a $d$-dimensional manifold and $\gamma$ be a closed one form on $N^d$, then we have
$H^k(M^n \times N^d, \pi_1^* \theta + \pi_2^*\gamma)\simeq \bigoplus_{p+q=k} H^p(M^n, \theta) \bigotimes H^q(N^d, \gamma)$,
 where $\pi_1: M^n \times N^d \rightarrow M^n, \pi_2: M^n \times N^d \rightarrow N^d$ are the projection maps.
\\
\\ (5) If $\pi: \widehat{M}^n \rightarrow M^n$ is a covering space with finite sheet,
then $\pi^*: H^p(M^n, \theta) \rightarrow H^p (\widehat{M}^n, \pi^* \theta)$ is injective for any $p$.

\end{thm}
\begin{proof}
See page 476-480 in \cite{HR} and Proposition 1.2 in \cite{O} for the proof of parts 1-4. For part 5, by Theorem \ref{iso}, we have
$$H^p(M^n, \theta) \simeq H^p(M^n, \mathbb{C}_{\rho}),$$
where $\mathbb{C}_{\rho}$ is the complex rank one local system defined by the linear representation
$$\rho: \pi_1(M^n) \rightarrow GL(1, \mathbb{C})= \mathbb{C}^*, [\gamma] \mapsto e^{\int_{\gamma} \theta} $$
and $H^p(M^n, \mathbb{C}_{\rho})$ is the $p$-th cohomology group of $M^n$ with coefficients in this local system.
\par As $\pi: \widehat{M}^n \rightarrow M^n$ is a covering space with finite sheet, one can construct a transfer map (see e.g. \cite{Go})
$h: H^p(\widehat{M}^n, \pi^*\mathbb{C}_{\rho} ) \rightarrow H^p (M^n, \mathbb{C}_{\rho})$ such that $h \pi^*=k Id$ , where
$k$ is the degree of $\pi$.
It follows that $\pi^*: H^p(M^n, \theta)\simeq H^p (M^n, \mathbb{C}_{\rho}) \rightarrow H^p(\widehat{M}^n, \pi^*\mathbb{C}_{\rho})
\simeq H^p(\widehat{M}^n, \pi^* \theta)$ is injective.

\end{proof}

As a corollary of Theorem \ref{basic}, we get
\begin{exm} \label{ex}
Let $M^n$ be $n$-dimensional torus, then
$H^p(M^n, \theta)=0$ for any $p$ and $[\theta] \neq 0$ by Theorem \ref{basic}.
\end{exm}

Let $\theta$ be a closed one form on $M^n$. Consider the following linear representation of the fundamental group of $M^n$:
$$\rho: \pi_1(M^n) \rightarrow GL(1, \mathbb{C})= \mathbb{C}^*, [\gamma] \mapsto e^{\int_{\gamma} \theta}. $$
The representation $\rho$ defines a complex rank one local system $\mathbb{C}_{\rho}$ over $M^n$ \cite{Di}. We denote
by  $H^p(M^n, \mathbb{C}_{\rho})$ the $p$-th cohomology group of $M^n$ with coefficients in this local system.

\begin{thm}  \label{iso}
$H^p(M^n, \theta) \simeq H^p(M^n, \mathbb{C}_{\rho})$ for any $p$.
\end{thm}
\begin{proof}
The proof is contained in \cite{Paz}.  For the convenience of the reader, we provide the details here.
Let $\pi : \widetilde{M}^n \rightarrow M^n$ be the universal cover of $M^n$.
The cohomology groups $H^p(M^n, \mathbb{C}_{\rho})$
are isomorphic to $H^p_{\rho}(\widetilde{M}^n)$, the cohomology groups of the complex $\Omega(\widetilde{M}^n, \rho)$, consisting of the $\rho$-equivariant
differential forms on $\widetilde{M}^n$ relative to the usual differential (the proof is analogous to the
sheaf-theoretic proof of de Rham's theorem). Let $h$ be a function on $\widetilde{M}^n$  such that
$dh =\pi^* \theta$. We give a mapping $F : \Omega^*(M^n) \rightarrow \Omega^*(\widetilde{M}^n, \rho)$ by the formula
$F(w)=e^h \pi^*w$. It is easy to see that $F$ is one-to-one and commutes with the differentials. Hence
$$H^p(M^n, \theta) \simeq H^p_{\rho}(\widetilde{M}^n) \simeq H^p(M^n, \mathbb{C}_{\rho}).$$
\end{proof}

\begin{thm} \label{h1}
Let $M^n$ be a $n$-dimensional manifold and $\theta$ a closed one form on $M^n$. If the fundamental group of $M^n$ has a finitely generated nilpotent subgroup of finite index,
then $H^1(M^n, \theta)=H^{n-1}(M^n, \theta)=0$ for any $[\theta] \neq 0$.
\end{thm}

\begin{proof}
Let $G \subseteq \pi_1(M^n)$ be a finitely generated nilpotent subgroup of finite index and $\pi:\widehat{M}^n \rightarrow M^n$ the covering
space of $M^n$ with
$\pi_1( \widehat{M}^n)\simeq G$. The closed one form
$\pi^*\theta$ defines a linear representation of $G$:
$$\rho: G=\pi_1(\widehat{M}^n) \rightarrow GL(1, \mathbb{C})= \mathbb{C}^*, [\gamma] \mapsto e^{\int_{\gamma} \pi^*\theta}. $$
The representation $\rho$ defines a complex rank one local system $\mathbb{C}_{\rho}$ over $\widehat{M}^n$.
We denote by  $H^p(\widehat{M}^n, \mathbb{C}_{\rho})$ the $p$-th cohomology group of $\widehat{M}^n$ with coefficients in the local system  $\mathbb{C}_{\rho}$.
Let $K(G,1)$ be the topological space such that $\pi_1(K(G,1))=G, \pi_i (K(G,1))=0, i \geq 2$ and $\mathbb{L}_{\rho}$ the
complex rank one local system  over $K(G,1)$ defined by $\rho$. Since the classifying map $\widehat{M}^n \rightarrow K(G,1)$ induces over $\mathbb{Q}$ a
cohomology isomorphism in degree one, we get
$$H^1(\widehat{M}^n, \mathbb{C}_\rho)\simeq H^1(K(G,1), \mathbb{L}_{\rho}).$$
As $\pi:\widehat{M}^n \rightarrow M^n$ is a finite cover, $[\theta] \neq 0$ implies that $[\pi^* \theta] \neq 0$.
Then $\mathbb{L}_{\rho}$ is a nontrivial local system over $K(G,1)$.
As $G$ is a finitely generated nilpotent group, by Theorem 2.2 in \cite{MP}, for any $p$, we have
$$H^p(K(G,1), \mathbb{L}_{\rho})=0.$$
In particular,
$$H^1(\widehat{M}^n, \mathbb{C}_\rho) \simeq H^1(K(G,1), \mathbb{L}_{\rho}) =0.$$
By Theorem \ref{basic} and Theorem \ref{iso}, we have
$$H^1(\widehat{M}^n, \pi^* \theta)=0$$
$$H^1(M^n, \theta)=0$$
$$H^{n-1}(M^n, \theta)\simeq H^1(M^n, -\theta)=0.$$
\end{proof}

\section{Cartan-Leray spectral sequence}
In this section we apply Cartan-Leray spectral sequence to prove Theorem \ref{nice}.
 By passing to a finite cover, we can assume that $M^n$ is a nilpotent space.
 The closed one form $\theta$ induces a  linear representation of $G=\pi_1(M^n)$:
$$\rho: \pi_1(M^n) \rightarrow GL(1, \mathbb{C})= \mathbb{C}^*, [\gamma] \mapsto e^{\int_{\gamma} \theta}. $$
By Theorem \ref{iso}, for any $p$, we have
$$
H^p(M^n, \theta)\simeq H^p(M^n, \mathbb{C}_{\rho}),
$$
where $\mathbb{C}_\rho$ is the complex rank one local system over $M^n$ defined by $\rho$.
By duality, it suffices to prove the vanishing of $H_p(M^n, \mathbb{C}_{\rho})$,
which is the homology group of $M^n$ with coefficients in the local system  $\mathbb{C}_{\rho}$.
Let $\widetilde{M}^n$ be the universal cover of $M^n$.
The representation $\rho$ together with the $G$ action on $\widetilde{M}^n$ by deck transformation induces the diagonal action on
 $H_l(\widetilde{M}^n,  \mathbb{C})\simeq H_l(\widetilde{M}^n, \mathbb{Z}) \otimes  \mathbb{C}$. By the Cartan-Leray spectral sequence
 ( Theorem 7.9, page 173 in \cite{Br}), we have
$$E^2_{kl}=H_k(G, H_l(\widetilde{M}^n,  \mathbb{C})) \Rightarrow H_{k+l}(M^n, \mathbb{C}_{\rho}),$$
where $H_k(G, H_l(\widetilde{M}^n, \mathbb{C}))$ is the $k$-th homology group of $G$ with coefficients in the $G$-module $H_l(\widetilde{M}^n,  \mathbb{C})$.
See \cite{Br} for more details of homology of groups. For us, we only need the following long exact sequence (Proposition 6.1, page 71 in \cite{Br}).
\begin{lem} \label{short}
For any short exact sequence $0 \rightarrow M' \rightarrow M \rightarrow M''\rightarrow 0$ of $G$-modules, there is the following long exact sequence:
$$\cdots \rightarrow H_i (G, M') \rightarrow H_i (G, M)  \rightarrow H_i (G, M'') \rightarrow H_{i-1} (G, M') \rightarrow H_{i-1} (G, M) \rightarrow \cdots$$
$$\rightarrow H_1 (G, M') \rightarrow H_1 (G, M) \rightarrow H_1 (G, M'')  \rightarrow H_0 (G, M') \rightarrow H_0 (G, M) \rightarrow H_0 (G, M'') \rightarrow 0.$$
\end{lem}
As $M^n$ is a nilpotent space, then
$G=\pi_1(M^n)$ is a nilpotent group that operates nilpotently on $\pi_m (M^n)$ for every $m \geq 2$. By Lemma 2.18 in \cite{HMR},
$G$ operates nilpotently on $H_l (\widetilde{M}^n, \mathbb{Z})$ for every $l$, that is $V=H_l(\widetilde{M}^n, \mathbb{Z})$
admits a finite sequence of $G$-invariant subgroups
$$V=V_0 \supseteq V_1 \supseteq \ldots V_k=0$$
such that the induced action of $G$ on $V_j / V_{j+1}$ is trivial for any $j$.
The representation $\rho$ of $G$ induces a diagonal action on $V_j \otimes \mathbb{C}$ and
we have the following short exact sequence of $G$ modules:
$$0 \rightarrow V_{j+1} \otimes \mathbb{C} \rightarrow V_j \otimes \mathbb{C} \rightarrow V_j /  V_{j+1} \otimes \mathbb{C} \rightarrow 0.$$
\par We now prove $H_k(G, V_j \otimes \mathbb{C})=0$  for any $j$ by induction. It is clear that $H_k(G, V_k \otimes \mathbb{C})=H_k(G, 0)=0.$
As $[\theta] \neq 0$, we see that $\rho$ is a nontrivial representation of $G$. By assumption, the induced action of $G$ on $V_j / V_{j+1}$ is trivial for any $j$.
Then the diagonal action of $G$ on $V_j /  V_{j+1} \otimes \mathbb{C}$ is nontrivial. As $G$ is a finitely generated nilpotent group,
by Theorem 2.2 in \cite{MP},  we get
$$H_k(G, V_j /  V_{j+1} \otimes \mathbb{C})=0.$$
By Lemma \ref{short} and induction, for any $j$, we get
$$H_k(G, V_j \otimes \mathbb{C})=0.$$
In particular,
$$H_k(G, H_l(\widetilde{M}^n,  \mathbb{C}))=H_k(G, V_0 \otimes \mathbb{C})=0.$$
By the Cartan-Leray spectral sequence \cite{Br}, we have
$$E^2_{kl}=H_k(G, H_l(\widetilde{M}^n,  \mathbb{C})) \Rightarrow H_{k+l}(M^n, \mathbb{C}_{\rho}).$$
Hence for any $k, l \geq 0$, we have
$$H_{k+l}(M^n, \mathbb{C}_{\rho})=0.$$
Then we get $H^p (M^n,\theta)=0$ for any $p$ and $[\theta] \neq 0.$

\section{An integral formula of $\Delta_{\theta}$ harmonic forms}
In section we derive an integral formula of $\Delta_{\theta}$ harmonic forms which will be crucial in the proof of Theorem \ref{ricci}.
\par Let $(M^n, g)$ be a closed Riemannian manifold and $\theta$ a closed real one form on $M^n$.
Define $d_{\theta}: \Omega^p(M^n) \rightarrow \Omega^{p+1}(M^n)$ as $d_{\theta} \alpha=d \alpha+ \theta \wedge \alpha$ for
$\alpha \in \Omega^p(M^n)$.
 Let $d^*$ be the formal $L^2$ adjoint of $d$ with respect to $g$. We can also define an operator $d_{\theta}^*$ as the formal $L^2$ adjoint of
$d_{\theta}$ with respect to $g$. Further, $\Delta_{\theta}=d_{\theta} d_{\theta}^*+ d_{\theta}^*d_{\theta}$ is the corresponding Laplacian.
These operators are lower-order perturbations of the corresponding operators in the usual Hodge-de Rham theory
and therefore have much the same analytic properties.
For example, the usual proof of the Hodge decomposition theorem goes through, and one obtains an orthogonal decomposition

$$\Omega^p(M^n)=\mathcal{H}^p(M^n) \oplus d_{\theta} (\Omega^{p-1}(M^n))\oplus d_{\theta}^* (\Omega^{p+1}(M^n)),$$
where $\mathcal{H}^p(M^n)$ is the space of $\Delta_{\theta}$ harmonic forms, which is isomorphic to $H^p(M^n, \theta)$.
\par Let $dV$ be the volume form of $g$ and $X$  the dual vector field of $\theta$
defined by $g(X, Y)=\theta (Y).$  Choose a $\Delta_{\theta}$ harmonic form $\alpha$ in $H^p(M^n, \theta)$. Then
$$d_{\theta}\alpha=d \alpha+  \theta \wedge \alpha=0$$
$$d_{\theta}^*\alpha=d^* \alpha+  i_X \alpha =0.$$

The following integral formula and its corollary \ref{key2} will be crucial in the proof of Theorem \ref{ricci}.

\begin{thm} \label{key}
$$\int_{M^n}  |X|^2 |\alpha|^2 dV=\frac{1}{2}\int_{M^n }\alpha \wedge[L_X, *]\alpha, $$
where $[L_X, *]\alpha=L_X*\alpha-*L_X\alpha$ and $L_X \alpha$ is the Lie derivative of $\alpha$ in the direction $X$.
\end{thm}

\begin{rem}
When $\theta$ is exact and $X=\nabla f$ for some smooth function $f$ on $M^n$, we believe that the integral formula in Theorem \ref{key} is the same as
\cite{DX}. It is also possible to adapt the method in \cite{DX} to prove Theorem \ref{key}. However, we present a different proof here.
\end{rem}

\begin{cor} \label{key2}
$$\int_{M^n} |X|^2 |\alpha|^2 dV \leq C_n \int_{M^n}  (|\nabla X|+ |div (X)|)|\alpha|^2 dV $$
for some constant $C_n$ depending only on $n$.
\end{cor}

\begin{proof}

The Riemannian metric $g$ on $M^n$ induces a linear map between $TM^n$ and $T^*M^n$ defined by
$$g: TM^n \rightarrow T^*M^n$$
$$<g(X), Y>=g(X,Y), \forall X, Y \in TM^n.$$
Let $g^{-1}$ be the inverse of the above map $g$ and $h$ the endomorphism of the bundle $T^*M^n \rightarrow M^n$ by
$$h=L_X g \circ g^{-1}.$$
The derivation of the Grassmann algebra $\Lambda T^*M^n$ induced by $h$ is denoted by $i(h)$. This is a linear map such that, if $\gamma \in T^*M^n$, then
$i(h)(\gamma)=h(\gamma)$, and
\begin{equation} \label{11}
i(h)(\omega_1 \wedge \omega_2)=(i(h)\omega_1) \wedge \omega_2 + \omega_1 \wedge (i(h)\omega_2)
\end{equation}
for any $\omega_1, \omega_2 \in \Lambda T^*M^n.$ The following formula is proved in \cite{T}.
\begin{equation} \label{star}
[L_X, *]\omega=(i(h)-\frac{1}{2}Tr h) *\omega
\end{equation}
for any $\omega \in \Lambda T^*M^n$.
\par
Let $div (X)$ be the divergence of $X$ with respect to $g$.  As
$$(L_X g) (Y, Z)= \ g(\nabla_Y X, Z)+  \ g (Y, \nabla_Z X)$$ for all $Y, Z \in TM^n$, we see that $Tr h= 2 div (X).$
Then by Theorem \ref{key}, we get
$$\int_{M^n} |X|^2 |\alpha|^2 dV \leq C_n \int_{M^n}  (|\nabla X|+ |div X|)|\alpha|^2 dV $$
for some constant $C_n$ depending only on $n$.
\end{proof}

Now we prove Theorem \ref{key}.
We firstly need the following lemmas.

\begin{lem} \label{com}
For any $p$ form $\omega$, we have
\begin{equation} \label{0}
*i_X \omega=(-1)^{p-1} \theta \wedge *\omega,
\end{equation}
where $*$ is the Hodge star operator with respect to $g$.
\end{lem}
\begin{proof}
For any $p-1$ form $\xi$, we have
$$\int_{M^n} \xi \wedge *i_{X} \omega=\int_{M^n} g (\xi, i_X \omega) dV$$
$$=\int_{M^n} g(\theta \wedge \xi, \omega) dV=\int_{M^n}  \theta \wedge \xi \wedge *\omega$$
$$=(-1)^{p-1} \int_{M^n}  \xi \wedge  \theta \wedge *\omega.$$
Hence
$$*i_X \omega=(-1)^{p-1} \theta \wedge *\omega.$$
\end{proof}

\begin{lem}
Let $\beta=*\alpha$, then
$$d \beta -  \theta \wedge \beta=0.$$
\end{lem}

\begin{proof}
As $d^*\alpha=(-1)^{n(p+1)+1}*d*\alpha$ and $d^*\alpha +i_X \alpha=0$, we get
$$(-1)^{n(p+1)+1}*d*\alpha+ i_X \alpha =0.$$
Hence
$$(-1)^{n(p+1)+1}**d*\alpha+ *i_X \alpha =0.$$
By Lemma \ref{com}, we have
$$*i_X \alpha=(-1)^{p-1} \theta \wedge *\alpha.$$
It follows that
$$(-1)^p d*\alpha + (-1)^{p-1}  \theta \wedge * \alpha=0$$
So
$$d \beta - \theta \wedge \beta=0.$$
\end{proof}

Now we proceed to prove Theorem \ref{key}.
As $d \alpha+  \theta \wedge \alpha=0$, we get
$$i_X d\alpha + i_X ( \theta \wedge \alpha)=0.$$
So
\begin{equation}\label{1}
i_X d\alpha \wedge \beta +  |X|^2 \alpha \wedge \beta -  \theta \wedge i_X \alpha  \wedge \beta=0.
\end{equation}
On the other hand, as $d \beta -  \theta \wedge \beta=0$, we get
$$i_X d\beta - i_X (\theta \wedge \beta)=0.$$
So
$$i_X d\beta \wedge \alpha-  |X|^2 \beta \wedge \alpha +  \theta \wedge i_X \beta \wedge \alpha=0.$$
Then
\begin{equation}\label{2}
 \alpha \wedge i_X d\beta -  |X|^2 \alpha \wedge \beta  + (-1)^p  \theta \wedge \alpha \wedge i_X \beta=0.
\end{equation}
By \ref{1}, \ref{2}, we get
\begin{equation} \label{3}
-i_X d\alpha \wedge \beta +  \alpha \wedge i_X d\beta -2  |X|^2 \alpha \wedge \beta + \theta \wedge i_X \alpha  \wedge \beta
 + (-1)^p  \theta \wedge \alpha \wedge i_X \beta=0.
\end{equation}
Combined with
$$\theta \wedge i_X \alpha \wedge \beta + (-1)^p \theta \wedge \alpha \wedge i_X \beta= \theta \wedge i_X (\alpha \wedge \beta)$$
$$=|X|^2\alpha \wedge \beta - i_X (\theta \wedge \alpha \wedge \beta)=|X|^2\alpha \wedge \beta,$$
we get
\begin{equation} \label{44}
-i_X d\alpha \wedge \beta +  \alpha \wedge i_X d\beta= |X|^2\alpha \wedge \beta.
\end{equation}
Since
$$d(i_X \alpha \wedge \beta)=d i_X \alpha \wedge \beta+ (-1)^{p-1} i_X \alpha \wedge d \beta,$$
we get
\begin{equation} \label{4}
\int_{M^n} i_X \alpha \wedge d\beta =(-1)^p \int_{M^n} d i_X \alpha \wedge \beta.
\end{equation}
On the other hand, we have
\begin{equation} \label {5}
0= i_X (\alpha \wedge d \beta)=i_X \alpha \wedge d \beta+ (-1)^p \alpha \wedge i_X d \beta.
\end{equation}
Combing \ref{4}, \ref{5}, we get
\begin{equation} \label{6}
\int_{M^n} \alpha \wedge i_X d\beta=-\int_{M^n} d i_X \alpha \wedge \beta.
\end{equation}
From \ref{44}, \ref{6}, we get
$$\int_{M^n} |X|^2\alpha \wedge \beta= -\int_{M^n} i_X d\alpha \wedge \beta -\int_{M^n}  d i_X \alpha \wedge \beta=
-\int_{M^n} L_X \alpha \wedge \beta$$
\begin{equation} \label{7}
=-\int_{M^n} L_X (\alpha \wedge \beta)+\int_{M^n} \alpha \wedge L_X \beta
=\int_{M^n} \alpha \wedge L_X \beta.
\end{equation}
As $\beta =* \alpha$, we get
\begin{equation}\label{8}
\int_{M^n} \alpha \wedge L_X \beta=\int_{M^n} \alpha \wedge L_X *\alpha
=\int_{M^n} \alpha \wedge * L_X \alpha+ \int_{M^n} \alpha \wedge [L_X, *]\alpha.
\end{equation}
Moreover,
$$\int_{M^n} \alpha \wedge * L_X \alpha=\int_{M^n} L_X \alpha \wedge *\alpha $$
$$=\int_{M^n} L_X (\alpha \wedge *\alpha)- \int_{M^n} \alpha \wedge  L_X *\alpha=-\int_{M^n}  \alpha \wedge  L_X *\alpha$$
$$=-\int_{M^n} \alpha \wedge  * L_X \alpha -\int_{M^n}  \alpha \wedge  [L_X, *]\alpha.$$
Hence
\begin{equation} \label{9}
\int_{M^n} \alpha \wedge * L_X \alpha=-\frac{1}{2} \int_{M^n} \alpha \wedge  [L_X, *]\alpha.
\end{equation}
By \ref{7}, \ref{8}, \ref{9}, we get
$$\int_{M^n}  |X|^2 |\alpha|^2 dV=\frac{1}{2}\int_{M^n }\alpha \wedge[L_X, *]\alpha.$$

\section{Proof of Theorem \ref{ricci}}
In this section we give a proof of Theorem \ref{ricci}.  The proof is based on Corollary \ref{key2}.
Another crucial tool is the following Poincar$\acute{e}$-Sobolev inequality (\cite{B}, page 397).

\begin{thm} \label{sob}
Let $(M^n,g)$ be a closed smooth Riemannian manifold such that for some constant $b>0$,
$$ r_{min}(g) D^2(g)\geq-(n-1)b^2,$$
where $D(g)$ is the diameter of $g$, $Ric(g)$ is the Ricci curvature of $g$ and
$$r_{min}(g)=inf\{{Ric(g)(u,u): u \in TM, g(u,u)=1}\}.$$
Let $R=\frac{D(g)}{b C(b)}$, where $C(b)$ is the unique positive root of the equation
$$ x \int_0^{b}(cht + x sht)^{n-1} dt=\int_0^{\pi} sin^{n-1}t dt.$$
Then for each $1 \leq p \leq \frac{nq}{n-q}, p < \infty$ and $f \in W^{1,q}(M^n)$, we have
$$ \|f-\frac{1}{V(g)}\int_{M^n}f dV \|_p \leq S_{p,q} \|df\|_q $$
$$ \|f\|_p \leq S_{p,q} \|df\|_q + V(g)^{1/p-1/q}\|f\|_q,$$
where $V(g)$ is the volume of $(M^n,g)$, $S(p,q)=(V(g)/vol(S^n(1))^{1/p-1/q} R \Sigma(n,p,q)$ and $\Sigma(n,p,q)$ is the Sobolev constant of
the canonical unit sphere $S^n$ defined by
$$\Sigma(n,p,q)=sup\{{\|f\|_p/\|df\|_q: f \in W^{1,q}(S^n), f \neq 0, \int_{S^n}f=0}\}.$$

\end{thm}

 Let $p=\frac{2n}{n-2}, q=2$ in Theorem \ref{sob} and apply Theorem 3 and Proposition 6 in \cite{B} pages 395-396, then we get the following mean value inequality.

\begin{thm}\label{sbl}
Let $ n\geq 3$ and $(M^n,g)$ be a closed $n$-dimensional smooth Riemannian manifold such that for some constant $b>0$,
$$r_{min}(g) D^2(g)\geq-(n-1)b^2.$$
If $f\in W^{1,2}(M^n)$ is a nonnegative continuous function such that $f \Delta f \geq -c f^2$ (here $\Delta$ is a negative operator)
in the sense of districution for some positive number $c$, then
$$max_{x \in M^n}|f|^2(x) \leq B_n (\sigma_n R c^{1/2}) \frac{\int_{M^n}f^2 dV}{V(g)},$$
where $\sigma_n=vol(S^n)^{1/n} \Sigma(n,\frac{2n}{n-2},2)$ and $B_n: \mathbb{R}_+ \rightarrow \mathbb{R}_+$ is a function defined by
$$B_n(x)=\prod_{i=0}^{\infty}(x \nu^i (2\nu^i-1)^{-1/2}+1)^{2\nu^{-i}}, \nu=\frac{n}{n-2}.$$
The function $B_n$ satisfies the inequalities
$$ B_n(x)\leq exp(2x\sqrt{\nu}/(\sqrt{\nu}-1)),  0 \leq x \leq 1$$
$$ B_n(x) \leq B_n(1)x^{2\nu/(\nu-1)},  x\geq 1.$$
In particular, $\lim_{x \rightarrow 0_{+}B_n(x)}=1$ and $B_n(x) \leq B_n(1) x^n$ for $x \geq 1$.
 \end{thm}

 Let $M^n$ be a closed Riemannian manifold with nonzero first de Rham
cohomology group and admits a sequence of Riemannian metrics $g_i$ such that
$$Ric (g_i) \geq -\frac{n-1}{i}$$
$$D(g_i) \leq 1.$$
Moreover, the curvature operator of $g_i$ is uniformly bounded from below by $-Id$. For any $[\theta] \in H^1_{dR}(M^n), [\theta] \neq 0$,
we are going to prove that there exists some $t \in \mathbb{R}, t \neq 0$ such that $H^p (M^n, t \theta)=0$ for any $p$.
If $n=2$, since the first Betti number of $M^2$ is bounded by $2$ (see e.g. \cite{B}), the genus of $M^2$ is at most $1$ and $H^p (M^2, t \theta)=0$ by Example \ref{ex}.
Now we assume that $n \geq 3.$
 Let $d^*$ be the formal $L^2$ adjoint of $d$ with respect to $g_i$.
 By Hodge theory, we can choose a harmonic one form $\theta_i$ in the cohomology class $[\theta]$. Then
 $$d \theta_i=0$$
 $$ d^*\theta_i=0$$
 $$\theta_i \neq 0.$$

 Let $t_i=(\frac{V(g_i)}{\int_{M^n} |X_i|^2 dV_i})^{1/2}>0$, where $V(g_i)$ is the volume of $(M^n, g_i)$, $dV_i$ is the volume form of $g_i$,
 $|X_i|^2=g_i(X_i,X_i)$ and $X_i$ is the dual vector field of $\theta_i$
defined by $g_i (X_i, Y)=\theta (Y).$
We claim that for sufficiently large $i$, $H^p(M^n, t_i \theta_i)=0$ for any $p$. Choose a $\Delta_{t_i\theta_i}$ harmonic form $\alpha_i$ in $H^p(M^n, t_i \theta_i)$. Then
$$d \alpha_i+ t_i \theta_i \wedge \alpha_i=0$$
$$d^* \alpha_i+  i_{t_i X_i} \alpha_i =0.$$

The goal is to prove that $\alpha_i \equiv 0.$  As $Ric(g_i) \geq  - \frac{n-1}{i}$, applying Bochner formula to $X_i$ \cite{P}, we get
\begin{equation} \label{413}
\frac{1}{2}\Delta |X_i|^2=|\nabla X_i|^2 + Ric(g_i)(X_i,X_i)\geq |\nabla X_i|^2 - \frac{n-1}{i}|X_i|^2,
\end{equation}
where $\Delta$ is the Laplacian acting on functions which is a negative operator.
Then
\begin{equation} \label{20}
\int_{M^n} |\nabla X_i|^2 dV_i \leq \frac{n-1}{i} \int_{M^n} |X_i|^2 dV_i.
\end{equation}

Let $div (X_i)$ be the divergence of $X_i$ with respect to $g_i$. As $\theta_i$ is a harmonic one form, we see $div (X_i)=0$ (see e.g. Proposition 31 in \cite{P} page 206).
By  Corollary \ref{key2}, we have
\begin{equation} \label{small}
\int_{M^n} t_i^2 |X_i|^2 |\alpha_i|^2 dV_i \leq C_n \int_{M^n} t_i |\nabla X_i||\alpha_i|^2 dV_i.
\end{equation}
for some constant $C_n$ depending only on $n$. Applying H$\ddot{o}$lder's inequality on \ref{small} and using \ref{20}, we get
$$\int_{M^n} t_i^2 |X_i|^2 |\alpha_i|^2 dV_i \leq C_n \int_{M^n}  t_i|\nabla X_i||\alpha_i|^2 dV_i $$
$$\leq C_n (\int_{M^n} t_i^2 |\nabla X_i|^2 dV_i )^{\frac{1}{2}} (\int_{M^n}|\alpha_i|^4 dV_i)^{\frac{1}{2}} $$
\begin{equation} \label{21}
\leq \frac{C_n}{\sqrt{i}} |\alpha_i|_{\infty} (\int_{M^n} t_i^2 |X_i|^2 dV_i )^{\frac{1}{2}} (\int_{M^n}|\alpha_i|^2 dV_i)^{\frac{1}{2}},
\end{equation}
where $|\alpha_i|_{\infty}=£ºmax_{x \in M^n}|\alpha_i|(x).$

\begin{lem} \label{es}
\begin{equation} \label{211}
|X_i|^2_{\infty}=:max_{x \in M^n}|X_i|^2(x) \leq B_n( \sigma_n R_i  \sqrt{\frac{n-1}{i}}) \frac{\int_{M^n} |X_i|^2 dV_i}{V(g_i)},
\end{equation}
\begin{equation} \label{212}
|\alpha_i|^2_{\infty}=:max_{x \in M^n}|\alpha_i|^2(x) \leq B_n( \sigma_n R_i (t_i^2 |X_i|^2_{\infty}+ C_n) ^{\frac{1}{2}}) \frac{\int_{M^n} |\alpha_i|^2 dV_i}{V(g_i)},
\end{equation}
where $R_i =\frac{D(g_i)}{\frac{1}{\sqrt{i}} C(\frac{1}{\sqrt{i}})}$, $C(\frac{1}{\sqrt{i}}), \sigma_n, B_n(x)$ are defined in Theorem \ref{sob} and Theorem \ref{sbl}
and $C_n$ is a positive constant depending only on $n$.
\end{lem}

\begin{proof}
Since $\theta_i$ is a harmonic one form, $div X_i=0$.
As $Ric(g_i) \geq  - \frac{n-1}{i}$, applying Bochner formula to $X_i$, we get
\begin{equation} \label{413}
\frac{1}{2}\Delta |X_i|^2=|\nabla X_i|^2 + Ric(g_i)(X_i,X_i)\geq |\nabla X_i|^2 - \frac{n-1}{i}|X_i|^2,
\end{equation}
where $\Delta$ is the Laplacian acting on functions which is a negative operator.
On the other hand, by Kato's inequality \cite{B}, we have $|\nabla X_i|\geq |\nabla|X_i||$. It follows that
\begin{equation} \label{414}
|X_i| \Delta |X_i| \geq  -\frac{n-1}{i}|X_i|^2.
\end{equation}
Since $Ric(g_i) \geq  - \frac{n-1}{i}$, $D(g_i) \leq 1$, we have
$$r_{min}(g_i) D^2(g_i) \geq -\frac{n-1}{i}.$$
Apply Theorem \ref{sbl} to $|X_i|$, we get
\begin{equation} \label{17}
|X_i|^2_{\infty}=:max_{x \in M^n}|X_i|^2(x) \leq B_n( \sigma_n R_i  \sqrt{\frac{n-1}{i}}) \frac{\int_{M^n} |X_i|^2 dV_i}{V(g_i)},
\end{equation}
where $R_i =\frac{D(g_i)}{\frac{1}{\sqrt{i}} C(\frac{1}{\sqrt{i}})}$.
Since the curvature operator of $g_i$ is bounded from below by $-Id$, applying Bochner formula to $\alpha_i$ \cite{P}, we get
\begin{equation} \label{13}
\frac{1}{2}\Delta |\alpha_i|^2 \geq |\nabla \alpha_i|^2  - |d\alpha_i|^2 - |d^*\alpha_i|^2 - C_n |\alpha_i|^2
\end{equation}
for some positive constant $C_n$ depending only on $n$.

\begin{lem} \label{os}
$$t_i^2 |X_i|^2 |\alpha_i|^2=|d\alpha_i|^2 + |d^*\alpha_i|^2.$$
\end{lem}

\begin{proof}
Firstly, we have
$$t_i^{2}|X_i|^{2}|\alpha_i|^{2} d V_i= t_i \theta_i \wedge i_{t_i X_i} (\alpha_i \wedge * \alpha_i)$$
$$=t_i^2 \theta_i \wedge i_{X_i} \alpha_i \wedge * \alpha_i + (-1)^p t_i^2 \theta_i \wedge \alpha_i \wedge i_{X_i} (* \alpha_i)$$
\begin{equation} \label{13.5}
=(-1)^{p-1} t_i^2  i_{X_i} \alpha_i \wedge \theta_i \wedge * \alpha_i + (-1)^p t_i^2 \theta_i \wedge \alpha_i \wedge i_{X_i} (* \alpha_i).
\end{equation}
By Lemma \ref{com}, we get
\begin{equation}
*i_{X_i} \alpha_i=(-1)^{p-1}\theta_i \wedge * \alpha_i;
\end{equation}
\begin{equation}
*i_{X_i} (*\alpha_i)=(-1)^{n-p-1}\theta_i \wedge ** \alpha_i=(-1)^{n-p-1}(-1)^{np+p}\theta_i \wedge  \alpha_i.
\end{equation}
Hence
\begin{equation}\label{14}
\theta_i \wedge * \alpha_i=(-1)^{p-1} *i_{X_i} \alpha_i
\end{equation}
\begin{equation} \label{15}
i_{X_i} (*\alpha_i)=(-1)^{n(n-p-1)+n-p-1} **i_{X_i} (*\alpha_i)=(-1)^p *(\theta_i \wedge  \alpha_i).
\end{equation}
By \ref{13.5}, \ref{14}, \ref{15}, we get
\begin{equation}
t_i^{2}|X_i|^{2}|\alpha_i|^{2} dV_i =t_i^2 i_{X_i} \alpha_i \wedge *(i_{X_i} \alpha_i)+ t_i^2 \theta_i \wedge \alpha_i \wedge *(\theta_i \wedge \alpha_i)
=\left(t_i^{2}\left|i_{X_i} \alpha_i\right|^{2}+t_i^{2}|\theta_i \wedge \alpha_i|^{2}\right) d V_i.
\end{equation}
Since $d \alpha_i + t_i \theta_i \wedge \alpha_i=0, d^* \alpha_i +i_{t_i X_i} \alpha_i=0$,
we get
$$t_i^2 |X_i|^2 |\alpha_i|^2=|d\alpha_i|^2 + |d^*\alpha_i|^2.$$
\end{proof}
Given Lemma \ref{os}, we have
\begin{equation} \label{18}
\frac{1}{2}\Delta |\alpha_i|^2 \geq |\nabla \alpha_i|^2  - t_i^2 |X_i|^2 |\alpha_i|^2 - C_n |\alpha_i|^2.
\end{equation}
By Kato's inequality, we have $|\nabla \alpha_i|\geq |\nabla|\alpha_i||$. It follows that
\begin{equation} \label{414}
|\alpha_i| \Delta |\alpha_i| \geq  -(t_i^2 |X_i|^2+ C_n ) |\alpha_i|^2 \geq -(t_i^2 |X_i|^2_{\infty}+ C_n ) |\alpha_i|^2.
\end{equation}
Apply Theorem \ref{sbl} to $|\alpha_i|$, we get
$$|\alpha_i|^2_{\infty}=:max_{x \in M^n}|\alpha_i|^2(x) \leq B_n( \sigma_n R_i (t_i^2 |X_i|^2_{\infty}+ C_n ) ^{\frac{1}{2}}) \frac{\int_{M^n} |\alpha_i|^2 dV_i}{V(g_i)}.$$
\end{proof}

\begin{lem} \label{bb}

\begin{equation} \label{22}
\frac{ \int_{M^n}  |X_i|^2 dV_i}{V(g_i)} \int_{M^n} |\alpha_i|^2 dV_i \leq  \int_{M^n}  |X_i|^2 |\alpha_i|^2 dV_i  +
\frac{2 C_n |\alpha_i|^2_{\infty}   }{\sqrt{i}}  R_i \sqrt{B_n( \sigma_n R_i  \sqrt{\frac{n-1}{i}})} \int_{M^n} |X_i|^2 dV_i
\end{equation}
for some constant $C_n$ depending only $n$.
\end{lem}
\begin{proof}
Let $h_i=|X_i|^2$ and $\overline{h_i}= \frac{\int_{M^n}  |X_i|^2 dV_i }{V(g_i)}$. By Theorem \ref{sob} in the case $p=q=2$, we get
$$\int_{M^n} |h_i - \overline{h_i}||\alpha_i|^2 dV_i \leq |\alpha_i|^2_{\infty} (\int_{M^n} |h_i - \overline{h_i}|^2 dV_i)^{\frac{1}{2}} (V(g_i))^{\frac{1}{2}}$$
$$\leq C_n  |\alpha_i|^2_{\infty} R_i  (\int_{M^n} |\nabla h_i|^2 dV_i)^{\frac{1}{2}} (V(g_i))^{\frac{1}{2}}$$
$$=2 C_n  |\alpha_i|^2_{\infty} R_i  (\int_{M^n} |X_i|^2 |\nabla |X_i||^2| dV_i)^{\frac{1}{2}} (V(g_i))^{\frac{1}{2}} $$
$$\leq 2 C_n  |\alpha_i|^2_{\infty} R_i (\int_{M^n} |X_i|^2 |\nabla X_i|^2 dV_i)^{\frac{1}{2}} (V(g_i))^{\frac{1}{2}} $$
$$\leq  2 C_n  |\alpha_i|^2_{\infty} R_i |X_i|_{\infty} (V(g_i))^{\frac{1}{2}}  (\int_{M^n} |\nabla X_i|^2 dV_i)^{\frac{1}{2}} $$
$$\leq  2 C_n  |\alpha_i|^2_{\infty} R_i \sqrt{B_n( \sigma_n R_i  \sqrt{\frac{n-1}{i}})} (\int_{M^n} |X_i|^2 dV_i)^{\frac{1}{2}}  (\int_{M^n} |\nabla X_i|^2 dV_i)^{\frac{1}{2}}$$
$$\leq  \frac{2 C_n |\alpha_i|^2_{\infty}  } {\sqrt{i}} R_i \sqrt{B_n( \sigma_n R_i  \sqrt{\frac{n-1}{i}})} \int_{M^n} |X_i|^2 dV_i.$$
 It follows that
$$ \frac{ \int_{M^n}  |X_i|^2 dV_i}{V(g_i)} \int_{M^n} |\alpha_i|^2 dV_i \leq  \int_{M^n}  |X_i|^2 |\alpha_i|^2 dV_i  +
\frac{2 C_n |\alpha_i|^2_{\infty}  } {\sqrt{i}} R_i \sqrt{B_n( \sigma_n R_i  \sqrt{\frac{n-1}{i}})} \int_{M^n} |X_i|^2 dV_i. $$

\end{proof}

\begin{lem}
Let $C(b)$ be the function defined in Theorem \ref{sob}. Namely, $C(b)$ is the unique positive root of the equation
$$ x \int_0^{b}(cht + x sht)^{n-1} dt=\int_0^{\pi} sin^{n-1}t dt.$$
Then
\begin{equation} \label{28}
 \liminf_{b \rightarrow 0} b C(b) \geq
  a_n>0
\end{equation}
for some constant $a_n$ depending only on $n$.
\end{lem}

\begin{proof}
Let $\omega_n=\int_0^{\pi} sin^{n-1}t dt.$ Then
$$\omega_n = C(b) \int_0^{b}(cht + C(b) sht)^{n-1} dt = C(b) \int_0^{b}(\frac{e^t + e^{-t}}{2} + C(b) \frac{e^t-e^{-t}}{2})^{n-1} dt \geq  C(b) b.$$
On the other hand, for any sequence $b_i \rightarrow 0$, we have
$$\omega_n= C(b_i) \int_0^{b_i}(\frac{e^t + e^{-t}}{2} + C(b_i) \frac{e^t-e^{-t}}{2})^{n-1} dt  $$
$$ \leq C(b_i) \int_0^{b_i} (\frac{e+ e^{-1}}{2}+ C(b_i) \frac{e^t-e^{-t}}{2})^{n-1} dt  $$
$$\leq C(b_i) b_i  (\frac{e+ e^{-1}}{2}+  2 b_i C(b_i))^{n-1}$$
$$\leq C(b_i) b_i (\frac{e+ e^{-1}}{2}+  2 \omega_n)^{n-1}$$
Hence for some constant $a_n$ depending only on $n$, we have
$$ \liminf_{b \rightarrow 0} b C(b) \geq a_n>0$$

\end{proof}

By \ref{21}, \ref{211}, \ref{212} and \ref{22}, we get
$$\frac{ \int_{M^n} t_i^2 |X_i|^2 dV_i}{V(g_i)} \int_{M^n} |\alpha_i|^2 dV_i \leq  \int_{M^n}  t_i^2 |X_i|^2 |\alpha_i|^2 dV_i  +
\frac{2 C_n |\alpha_i|^2_{\infty}   }{\sqrt{i}} R_i \sqrt{B_n( \sigma_n R_i  \sqrt{\frac{n-1}{i}})}  \int_{M^n} t_i^2 |X_i|^2 dV_i.$$
$$\leq \frac{C_n}{\sqrt{i}} |\alpha_i|_{\infty} (\int_{M^n} t_i^2 |X_i|^2 dV_i )^{\frac{1}{2}} (\int_{M^n}|\alpha_i|^2 dV_i)^{\frac{1}{2}}+
\frac{2 C_n |\alpha_i|^2_{\infty}  }{\sqrt{i}} R_i \sqrt{B_n( \sigma_n R_i  \sqrt{\frac{n-1}{i}})}  \int_{M^n} t_i^2 |X_i|^2 dV_i$$
$$\leq \frac{C_n \sqrt{B_n( \sigma_n R_i (t_i^2 |X_i|^2_{\infty}+ C_n)^{\frac{1}{2}})}}{\sqrt{i}} \sqrt{\frac{\int_{M^n} t_i^2 |X_i|^2 dV_i } {V(g_i)}}
 \int_{M^n}|\alpha_i|^2 dV_i $$
\begin{equation} \label{final}
+ \frac{2 C_n  B_n(\sigma_n R_i (t_i^2 |X_i|^2_{\infty}+C_n)^{\frac{1}{2}})}{\sqrt{i}} R_i \sqrt{B_n( \sigma_n R_i  \sqrt{\frac{n-1}{i}})}  \frac{ \int_{M^n} t_i^2 |X_i|^2 dV_i} {V(g_i)} \int_{M^n} |\alpha_i|^2 dV_i,
\end{equation}
where
$$|X_i|^2_{\infty}=:max_{x \in M^n}|X_i|^2(x) \leq B_n( \sigma_n R_i {\sqrt\frac{n-1}{i}}) \frac{\int_{M^n} |X_i|^2 dV_i}{V(g_i)}.$$

As $t_i=(\frac{V(g_i)}{\int_{M^n} |X_i|^2 dV_i})^{1/2}$, we see
\begin{equation} \label{29}
\frac{\int_{M^n} t_i^2 |X_i|^2 dV_i}{V(g_i)}=1.
\end{equation}
Recall that $R_i =\frac{D(g_i)}{\frac{1}{\sqrt{i}} C(\frac{1}{\sqrt{i}})}$ and $D(g_i) \leq 1$.
By \ref{28}, \ref{final} and \ref{29}, using the properties of $B_n(x)$ in Theorem \ref{sbl}, we see that for sufficiently large $i$,
$$ \int_{M^n} |\alpha_i|^2 dV_i \leq  \frac{1}{2}\int_{M^n} |\alpha_i|^2 dV_i.$$
Hence $\alpha_i \equiv 0$ and $H^p (M^n, t_i \theta_i )=0$ when $n \geq 3.$

\end{document}